\documentclass{amsart}
\usepackage[margin=3.5cm]{geometry}

\usepackage{verbatim}
\usepackage{mathtools}
\usepackage{cite}
\usepackage{amssymb}
\usepackage[english]{babel}
\usepackage{todonotes}

\usepackage{xcolor}

\usepackage{hyperref}
\definecolor{darkgreen}{rgb}{0,0.4,0}
\definecolor{BrickRed}{rgb}{0.65,0.08,0}
\hypersetup{colorlinks=true,linkcolor=blue,citecolor=red,filecolor=BrickRed,urlcolor=darkgreen}

\newtheorem{theorem}{Theorem}[section]
\newtheorem{conj}{Conjecture}
\newenvironment{conjecture}[1]
  {\renewcommand\theconj{#1}\conj}
  {\endconj}
\newtheorem{proposition}[theorem]{Proposition}
\newtheorem{corollary}[theorem]{Corollary}
\newtheorem{lemma}[theorem]{Lemma}
\theoremstyle{remark}
\newtheorem*{remark}{Remark}
\numberwithin{equation}{section}
\newcommand{\dN}{\mathbb N}                             
\newcommand{\dZ}{\mathbb Z}                             
\newcommand{\dC}{\mathbb C}                             
\newcommand{\abs}[1]{\left\lvert#1\right\rvert}         
\newcommand{\comp}{\textrm{c}}                           

\DeclareRobustCommand{\TuDengConj}{\textup{TD}}
\DeclareRobustCommand{\CusickConj}{\textup{C}}
\DeclareRobustCommand{\CusickConjComp}{\textup{CC}}
\DeclareMathOperator{\dens}{\,\mathrm dens}            
\DeclareMathOperator{\LandauO}{\mathcal O}             

\allowdisplaybreaks
\begin{document}
\title[]%
{The Tu--Deng Conjecture holds almost surely}
\author{Lukas Spiegelhofer}
\author{Michael Wallner}
\address{Institut f\"ur diskrete Mathematik und Geometrie,
Technische Universit\"at Wien,
Wiedner Hauptstrasse 8--10, 1040 Wien, Austria}


\begin{abstract}
The Tu--Deng Conjecture is concerned with the sum of digits $w(n)$ of $n$ in base~$2$ (the Hamming weight of the binary expansion of $n$) and states the following:
assume that $k$ is a positive integer and $t\in\{1,\ldots,2^k-2\}$.
Then
\[\Bigl \lvert\Bigl\{(a,b)\in\bigl\{0,\ldots,2^k-2\bigr\}^2:a+b\equiv t\bmod 2^k-1, w(a)+w(b)<k\Bigr\}\Bigr \rvert\leq 2^{k-1}.\]

We prove that the Tu--Deng Conjecture holds almost surely in the following sense:
the proportion of $t\in\{1,\ldots,2^k-2\}$ 
such that the above inequality holds approaches $1$ as $k\rightarrow\infty$.

Moreover, we prove that the Tu--Deng Conjecture implies a conjecture due to T.~W.~Cusick concerning the sum of digits of $n$ and $n+t$.
\end{abstract}
\maketitle
\section{Introduction and results}
Z.~Tu and Y.~Deng's Conjecture~\cite{TD2011}
is concerned with the Hamming weight $w(n)$ of the binary expansion of a nonnegative integer $n$ (the sum of digits of $n$ in base two) and addition modulo $2^k-1$.
This conjecture is as follows.
\begin{conjecture}{\TuDengConj}\label{conj:TD}
Assume that $k$ is a positive integer and
$t\in\bigl\{1,\ldots,2^k-2\bigr\}$.
Define
\[S_{t,k}=\Bigl\{(a,b)\in\bigl\{0,\ldots,2^k-2\bigr\}^2:a+b\equiv t\bmod 2^k-1, w(a)+w(b)<k\Bigr\}.\]
Then $P_{t,k}\coloneqq \lvert S_{t,k}\rvert/2^k\leq 1/2$.
\end{conjecture}
The conjecture arose in the construction of Boolean functions with optimal algebraic immunity (see Tu and Deng~\cite{TD2011,TD2012}).
Indeed, if the conjecture is true, the functions defined by Tu and Deng have this property.

Such functions are used in the construction of stream ciphers, which are widely used encryption methods due to their high speed and low hardware requirements~\cite{C2010}. 
However, they are prone to serious attacks~\cite{C2003,CM2003,A2004}. 
In order to prevent them from these known attacks algebraic immunity was introduced~\cite{MPC2004}. 
We refer the reader to the above-cited papers by Tu and Deng for a more extensive discussion of the r\^ole of their conjecture within the cryptographic context.

So far the conjecture could only be solved for some special cases~\cite{CLS2011,DY2012,FRCM10,QSF2016}.
Moreover, it was checked for all $k\leq 29$ by Tu and Deng~\cite{TD2011}
and for $k\in\{39,40\}$ by Flori~\cite{F2012}.

Let us give a probabilistic (and combinatorial) interpretation of the conjecture. 
Let $S_k := \bigcup_{t=1}^{2^k-2} S_{t,k}$. 
Let us consider an arbitrary pair $(a,b)$ of $S_k$.
On the one hand, the number of $1$s in the binary expansion of $a$ (and $b$) is at most $k-1$  . 
On the other hand, the constraint on the Hamming weights implies that the total number of $1$s in both integers is less than $k$. 
Finally, note that all such pairs except $(0,0)$ are part of $S_k$.
Therefore, considering how we may (or actually may not) distribute $1$s on the $2k$ digits in base $2$ of $a$ and $b$ together we get
\begin{align*}
	|S_k| =  2^{2k} - \sum_{i=k}^{2k} {2k \choose i} - 1 
        = \frac{1}{2} \left(2^{2k} - {2k \choose k}\right) - 1.
\end{align*}
The sequence including $(0,0)$, i.e., the sequence for $|S_k|+1$ is \texttt{A000346} in Sloane's OEIS\footnote{\url{http://oeis.org}}. 

It is then easy to compute the asymptotic expansion of this sequence as
\[ |S_k|  = \frac{2^{2k}}{2} \left(1 - \frac{1}{\sqrt{\pi k}} + O\left(\frac{1}{k^{3/2}}\right)\right). \]
As there are $2^k-2$ possible choices for $t$ we see by the pigeonhole principle that at least one of the sets $S_{t,k}$ has to be asymptotically of size $2^k/2$.
Therefore, the Tu--Deng Conjecture describes a uniform distribution among the possible sets $S_{t,k}$.

While working on the Tu--Deng Conjecture,
T.~W.~Cusick (private communication, 2011, 2015) formulated a related conjecture on the Hamming weight:
\begin{conjecture}{\CusickConj}\label{conj:C}
Assume that $t$ is a nonnegative integer.
Then
\[c_t\coloneqq\dens\bigl\{n\in\dN:w(n+t)\geq w(n)\bigr\}>\frac 12,\]
where $\dens A$ denotes the asymptotic density of a set $A\subseteq \dN$
(which exists in this case).
\end{conjecture}


Also, note that the density in Conjecture~\ref{conj:C} exists, which follows, for example,
from the ``Lemma of B\'esineau''~\cite[Lemme~1]{B1972},
see also~\cite[Lemma~2.1]{DKS2016}.
In fact, we have
\begin{align}
c_t&=\frac 1{2^k}\bigl \lvert \{n<2^k:w(n+t)\geq w(n)\}\bigr \rvert
\label{eqn:ct_finite}
\end{align}
for $k\geq \alpha+\mu$, where $\alpha=w(t)+1$ and $2^\mu\leq t<2^{\mu+1}$~\cite[equation~(10) and Section~3.3]{DKS2016}.
We also studied~\cite{DKS2016} a statement complementary to Cusick's Conjecture:
\begin{conjecture}{\CusickConjComp}\label{conj:C2}
Assume that $t$ is a nonnegative integer. Then 
\[\tilde c_t\coloneqq\dens\bigl\{n\in\dN:w(n+t)> w(n)\bigr\}\leq \frac 12.\]
\end{conjecture}
Analogously to the case $c_t$, we have
\begin{align}
\tilde c_t&=\frac 1{2^{k-1}}\bigl \lvert \{n<2^{k-1}:w(n+t)>w(n)\}\bigr \rvert.
\label{eqn:tildect_finite}
\end{align}
for $k$ large enough.
Taken together, Conjectures~\ref{conj:C} and~\ref{conj:C2}
locate quite precisely the median of the random variable $X_t$ on $\dZ$ defined by
\[j\mapsto \dens\bigl\{n:w(n+t)-w(n)=j\bigr\}.\]
Numerical experiments reveal that $\tilde c_t\leq 1/2<c_t$ for all $t<2^{30}$.
In fact, Drmota, Kauers, and the first author~\cite{DKS2016} proved that Conjectures~\ref{conj:C} and~\ref{conj:C2} are satisfied for almost all $t$ in the sense of asymptotic density.
In the present paper, we want to show that an analogous result holds for Conjecture~\ref{conj:TD}.


\begin{theorem}\label{thm:main}
Define $P_{t,k}$ as before,
\[P_{t,k}=\frac{1}{2^k}\left\lvert\Bigl\{(a,b)\in\bigl\{0,\ldots,2^k-2\bigr\}^2:a+b\equiv t\bmod 2^k-1, w(a)+w(b)<k\Bigr\}\right\rvert.\]
For each $\varepsilon>0$, we have for $k\rightarrow\infty$
\[\bigl \lvert\bigl\{t\in\{1,\ldots,2^k-2\}:P_{t,k}\not\in(1/2-\varepsilon, 1/2)\bigr\}\bigr \rvert =O\left(\frac{2^k}k\right). \]
In particular,
\[
\lim_{k\rightarrow\infty}\frac 1{2^k}\bigl \lvert\bigl\{t\in\{1,\ldots,2^k-2\}:
1/2-\varepsilon<P_{t,k}< 1/2\bigr\}\bigr \rvert =1.
\]
\end{theorem}

Moreover, we will prove that Conjectures~\ref{conj:C} and~\ref{conj:C2} are in fact implied by Conjecture~\ref{conj:TD}.
\begin{proposition}\label{prp:implication}
Conjecture~\ref{conj:TD} implies Conjectures~\ref{conj:C} and~\ref{conj:C2}.
\end{proposition}
In fact, we will see that Conjectures~\ref{conj:C} and~\ref{conj:C2} are contained as ``extremal cases'' in Conjecture~\ref{conj:TD}, choosing $t$ and letting $k\rightarrow \infty$.

However, so far we did not succeed in proving the opposite implication.
Meanwhile, due to the similarity of the conjectures,
it is reasonable to expect that a proof of Conjecture~\ref{conj:C},
when one is found (and if it is found first),
 will lead to a proof of Conjecture~\ref{conj:TD}.
We wish to highlight this similarity between the conjectures.

\begin{proposition}\label{prp_similar}
For integers $k\geq 1$ and $a,b$ we define
\[a\oplus_k b=(a+b)\bmod (2^k-1).\]
Conjecture~\ref{conj:TD} is equivalent to the statement that
\begin{equation}\label{eqn_TD_equiv}
\lvert\{n\in\{0,\ldots,2^k-1\}:w(n\oplus_k t)\geq w(n)\}\rvert\geq 2^{k-1}
\end{equation}
for all $k\geq 1$ and $t\in\{1,\ldots,2^k-2\}$.
Conjecture~\ref{conj:C} is equivalent to the statement that
\begin{equation}\label{eqn_C_equiv}
\lvert\{n\in\{0,\ldots,2^k-1\}:w(n+t)\geq w(n)\}\rvert>2^{k-1}
\end{equation}
for all $k,t\geq 1$.
\end{proposition}
The binary operation $\oplus_k$
can also be seen as ``circular addition'' in base $2$:
if a carry occurs at the index $k-1$ in the addition $a+b$, this carry does not propagate into position $k$, but into the lowest bit instead.
Moreover, if $a+b=2^k-1$, the result is set to zero.

By Proposition~\ref{prp_similar}, we may summarize the content of Conjectures~\ref{conj:TD} and~\ref{conj:C} by the following elementary question: how does the sum of digits change under (modular) addition of a constant?
It is this formulation in particular that makes the Tu--Deng Conjecture a mathematically interesting problem.

The idea of the proof of Theorem~\ref{thm:main} is to show a concentration result using Chebyshev's inequality.
More precisely, we consider the moments
\[\frac 1{2^k}\sum_{0\leq t<2^k}\lvert S_{t,k}\rvert\quad\textrm{and}\quad \frac 1{2^k}\sum_{0\leq t<2^k}\lvert S_{t,k}\rvert^2\]
and derive asymptotic expansions for them.
(Note that $|S_{0,k}|=|S_{2^k-1,k}|=1$,
so that the cases $t\in\{0,2^k-1\}$ will not matter asymptotically.)
These expansions are then used to prove that the values $P_{t,k}$ concentrate well below $1/2$, as $k\rightarrow\infty$.
This idea of proof is analogous to the method used by Drmota, Kauers, and the first author~\cite{DKS2016}.
In fact, the trivariate rational generating function we are going to encounter is very similar to the one in that paper.

The remaining part of this paper is dedicated to the proofs of Theorem~\ref{thm:main} and Propositions~\ref{prp:implication} and~\ref{prp_similar}.
Throughout the proofs, we will use the notation $t^\comp_k=2^k-1-t$.
We will assume that $0\leq t<2^k$; then the binary expansion of $t^\comp_k$ is the Boolean complement of the binary expansion of $t$, padded with $1$s up to the index $k-1$.

\section{Proof of Proposition~\ref{prp:implication}}
We first rewrite the Tu--Deng Conjecture.
Let us split the set $S_{t,k}$ according to whether $a+b<2^k-1$: set
\begin{align*}
S^{(1)}_{t,k}&=\bigl\{a\in\{0,\ldots,t\}:w(a)+w(t-a)< k\bigr\},\\
S^{(2)}_{t,k}&=\bigl\{a\in\bigl\{t+1,\ldots,2^k-2\bigr\}:
                      w(a)+w\bigl(2^k-1+t-a\bigr)< k\bigr\}.
\end{align*}
Note that the sets 
$M^{(1)}_{t,k} = \{ (a,t-a) : a \in S^{(1)}_{t,k} \}$ and 
$M^{(2)}_{t,k} = \{ (a,2^t-1+t-a) : a \in S^{(2)}_{t,k} \}$ form a partition of $S_{t,k}$.
We define the quantity
\[\beta_{t,k,j}=\bigl \lvert\bigl\{a\in\bigl\{0,\ldots,t\bigr\}: w\bigl(a+2^k-1-t\bigr)-w(a)=j\bigr \rvert,\]
where $k\geq 1$, $0\leq t<2^k$ and $j$ are integers.
By the identity $w\bigl(2^k-1-t\bigr)=k-w(t)$
we have
\[    S^{(1)}_{t,k}=\bigl\{a\in\{0,\ldots,t\}:
                           w(a)<w\bigl(a+2^k-1-t\bigr)\bigr\}    \]
and
\begin{align*}
S^{(2)}_{t,k}&=\bigl\{a\in\bigl\{t+1,\ldots,2^k-2\bigr\}:w(a)<w(a-t)\bigr\}\\
&=\bigl\{a\in\bigl\{0,\ldots,2^k-2-(t+1)\}:
         w\bigl(2^k-1-(a+1)\bigr)< w\bigl(2^k-1-(a+t+1)\bigr)\bigr\}\\
&=\bigl\{a\in\bigl\{1,\ldots,2^k-2-t\bigr\}:w(a)>w(a+t)\bigr\}.
\end{align*}

Since $w(0)\not >w(0+t)$ and $w\bigl(2^k-1-t\bigr)\not >w\bigl(2^k-1\bigr)$, we obtain
\begin{equation}\label{eqn:two_parts}
\begin{aligned}
\lvert S_{t,k}\rvert&=
\bigl \lvert S^{(1)}_{t,k}\bigr \rvert
+\bigl \lvert S^{(2)}_{t,k}\bigr \rvert
\\&=\bigl\vert \bigl\{a\in\bigl\{0,\ldots,t\bigr\}:w\bigl(a+2^k-1-t\bigr)>w(a)\bigr\}\bigr\vert
\\&\quad+
\bigl\vert \bigl\{a\in\bigl\{0,\ldots,2^k-1-t\bigr\}:w(a)>w(a+t)\bigr\}\bigr\vert
\\&=
\sum_{j\geq 1}
\Bigl(
\beta_{t,k,j}+
\beta_{2^k-1-t,k,-j}
\Bigr).
\end{aligned}
\end{equation}
Both Conjecture~\ref{conj:C} and Conjecture~\ref{conj:C2} are trivial if $t=0$.
Let $t\geq 1$ be given and assume that $k'\geq 1$ is such that $t<2^{k'}-1$;
we choose $k\geq 2k'$.
With this choice we have
$w(a)
     \leq w\bigl(a+2^k-1-t\bigr)$ as long as $0\leq a\leq t$.
This is the case since $2^k-2^{k'}+1\leq a+2^k-1-t\leq 2^k-1$, therefore 
the tail of $\mathtt 1$s at the left of the binary expansion of $2^k-1-t$, having length at least $k'$, is not touched by the addition of $a$.
Therefore $\bigl \lvert S^{(1)}_{t,k}\bigr \rvert=t+1$ for large $k$.
Assuming that Conjecture~\ref{conj:TD} holds,
we obtain
\begin{align*}
2^{k-1}
&\geq
t+1+\bigl \lvert
\bigl\{a\in\bigl\{0,\ldots,2^k-1-t\bigr\}:w(a)>w(a+t)\bigr\}
\bigr \rvert
\\&>
\bigl \lvert\bigl\{a\in\{0,\ldots,2^k-1\}:w(a)>w(a+t)\bigr\}\bigr \rvert
\end{align*}
This last expression equals $2^k\bigl(1-c_t\bigr)$ if $k$ is chosen large enough (see~\eqref{eqn:ct_finite}), which implies $c_t>1/2$.
To derive Conjecture~\ref{conj:C2}, we replace $t$ in the Tu--Deng Conjecture by $2^k-1-t$.
Noting that $\sum_{j\in\dZ}\beta_{t,k,j}=t+1$, we obtain
\begin{align*}
2^{k-1}\geq \lvert S_{2^k-1-t,k}\rvert 
&=\sum_{j\geq 1}
\bigl(\beta_{2^k-1-t,k,j}+\beta_{t,k,-j}\bigr)
\\&=
\bigl \lvert \{a\in\{0,\ldots,2^k-1-t\}:w(a+t)-w(a)>0\}\bigr \rvert+\LandauO(t)
\\&=
\bigl \lvert \{a\in\{0,\ldots,2^k-1\}:w(a+t)-w(a)>0\}\bigr \rvert+\LandauO(t).
\end{align*}
Letting $k\rightarrow\infty$ and using~\eqref{eqn:tildect_finite} we obtain $\tilde c_t\leq 1/2$.

\begin{remark}
The quantities $\beta_{t,k,j}$ are linked to divisibility by powers of two in Pascal's triangle:
We define (see e.g.~\cite{SW2017})
\begin{align*}
\vartheta(j,n)=\biggl \lvert
\biggl\{k\in\{0,\ldots,n\}:\nu_2\binom nk=j\biggr\}
\biggr \rvert.
\end{align*}
(Here $\nu_2(m)$ denotes the largest $j$ such that $2^j$ divides $m$.)
Then for $k\geq 1$, $0\leq t<2^k$ and $j\geq 0$
we have the identity
\begin{align*}
\beta_{t,k,k-w(t)-j}=\vartheta(j,t).
\end{align*}
\begin{proof}
By the identity $\nu_2(n!)=n-w(n)$ we have
$\nu_2\binom nk=w(k)+w(n-k)-w(n)$
for $0\leq k\leq n$.
By the substitution $a\mapsto t-a$ and the formula $w\bigl(2^k-1-m\bigr)=k-w(m)$, valid for $m<2^k$, we obtain
\begin{align*}
\beta_{t,k,k-w(t)-j}&=
\bigl \lvert\bigl\{a\in\{0,\ldots,t\bigr\}:w\bigl(2^k-1-t+a\bigr)-w(a)=k-w(t)-j\bigr\}\bigr \rvert\\
&=\bigl \lvert\bigl\{a\in\{0,\ldots,t\}:w\bigl(2^k-1-t+(t-a)\bigr)-w(t-a)=k-w(t)-j\bigr\}\bigr \rvert\\
&=\bigl \lvert\bigl\{a\in\{0,\ldots,t\}:w(a)+w(t-a)-w(t)=j\bigr\}\bigr \rvert\\
&=\vartheta(j,t).
\qedhere
\end{align*}
\end{proof}
\end{remark}
\section{Proof of Proposition~\ref{prp_similar}}
Using the identity $w(2^k-1-t)=k-w(t)$ (see also the proof of Proposition~\ref{prp:implication} from the previous section), we see that
\[\lvert S_{t,k}\rvert
=
\lvert\{a\in\{0,\ldots,t\}:w(a)<w(a+t^\comp_k)\}\rvert
+\lvert\{a\in\{t+1,\ldots,2^k-2\}:w(a)<w(a+t^\comp_k-2^k+1\}\rvert.
\]
We wish to replace addition by $\oplus_k$.
To do so, we note that
$w(t)<w(t+t^\comp_k)$, but $w(t)\nless w(t\oplus_k t^\comp_k)$.
It follows that 
\[
\lvert S_{t,k}\rvert = 1+\lvert\{a\in\{0,\ldots,2^k-2\}: w(a)<w(a\oplus_k t^\comp_k)\}\rvert
\]

For all $t\in\{1,\ldots,2^k-2\}$ and $a\in\{0,\ldots,2^k-2\}$ we have the identity
$(a\oplus_k t)\oplus_k t^\comp_k=a$,
therefore
\begin{align*}
\lvert S_{t,k}\rvert
&=
1+\lvert\{a\in\{0,\ldots,2^k-2\}: w(a\oplus_k t)<w(a)\}\rvert
\\&=
\lvert\{a\in\{0,\ldots,2^k-1\}: w(a\oplus_k t)<w(a)\}\rvert,
\end{align*}
where we used $w((2^k-1)\oplus_k t)<w(2^k-1)$.
From this the first equivalence follows.

In order to prove the second statement, it is sufficient to show that the values
$2^{-k}\bigl \lvert\{n\in\{0,\ldots,2^k-1\}:w(n+t)\geq w(n)\}\bigr \rvert$ are nonincreasing in $k$.
(Note that this is not the case for Tu--Deng; otherwise we would have a proof of the implication \ref{conj:C}$\Rightarrow$\ref{conj:TD}.)
We proceed by induction on $t$ and show the more general statement that the values
$v_{t,k,j}=2^{-k}\bigl \lvert\{n\in\{0,\ldots,2^k-1\}:w(n+t)-w(n)\geq j\}\bigr \rvert$ are nonincreasing in $k$, for each $j\in\mathbb Z$.

We first prove the statement for $t=1$, using the identity $w(n+1)-w(n)=1-\nu_2(n+1)$. Here $\nu_2(a)$ is the $2$-valuation of $a\geq 1$, that is, the largest $k$ such that $2^k\mid a$.
By this identity we have $v_{1,k,j}=0$ for $j\geq 2$.
Moreover, $\nu_2(n+1)\geq \ell$ if and only if the lowest $\ell$ digits of $n$ are $1$. Therefore we obtain

\[\lvert\{n\in\{0,\ldots,2^k-1\}:\nu_2(n+1)\geq \ell\}\rvert
=\begin{cases}
0,&k<\ell;\\
2^{k-\ell},&k\geq \ell
\end{cases}
\]
for all $\ell\geq 0$, and the statement follows.

In the following, we write $d(n,t)=w(n+t)-w(n)$ for brevity.
Assume that the statement holds for $t$; we wish to prove it for $2t$ and $2t+1$ in place of $t$.
We have $d(2n,2t)=d(2n+1,2t)=d(n,t)$, therefore
$v_{2t,0,j}=v_{2t,1,j}$ for all $j$.
Moreover, for $k\geq 0$ we get
\begin{align*}
\bigl \lvert\{n\in\{0,\ldots,2^{k+1}-1\}:d(n,2t)\geq j\}\bigr \rvert
&=\bigl \lvert\{2n:n\in\{0,\ldots,2^k-1\},d(2n,2t)\geq j\}\bigr \rvert
\\&+\bigl \lvert\{2n+1:n\in\{0,\ldots,2^k-1\},d(2n+1,2t)\geq j\}\bigr \rvert
\\&=2\bigl \lvert\{n\in\{0,\ldots,2^k-1\},d(n,t)\geq j\}\bigr \rvert,
\end{align*}
therefore $v_{2t,k+1,j}=v_{t,k,j}\leq v_{t,k-1,j}=v_{2t,k,j}$ for $k\geq 1$.

It remains to treat the case $2t+1$.
We have $d(0,2t+1)=w(2t+1)-w(0)=w(t)+1$
and $d(1,2t+1)=w(2t+2)-w(1)=w(t+1)-1$,
which implies $v_{2t+1,0,j}\geq v_{2t+1,1,j}$ by the inequality $w(n+1)\leq w(n)+1$.

Moreover, we have $d(2n,2t+1)=d(n,t)+1$ and $d(2n+1,2t+1)=d(n,t+1)-1$,
therefore we have for all $k\geq 0$
\begin{align*}
\hspace{5em}&\hspace{-5em}
\bigl \lvert\{n\in\{0,\ldots,2^{k+1}-1\}:d(n,2t+1)\geq j\}\bigr \rvert
\\
&=\bigl \lvert\{2n:n\in\{0,\ldots,2^k-1\},d(2n,2t+1)\geq j\}\bigr \rvert
\\&\quad+\bigl \lvert\{2n+1:n\in\{0,\ldots,2^k-1\},d(2n+1,2t+1)\geq j\}\bigr \rvert\\
&=\bigl \lvert\{2n:n\in\{0,\ldots,2^k-1\},d(n,t)\geq j-1\}\bigr \rvert
\\&\quad+\bigl \lvert\{2n+1:n\in\{0,\ldots,2^k-1\},d(n,t+1)\geq j+1\}\bigr \rvert.
\end{align*}
It follows that 
$v_{2t+1,k+1,j}=\frac 12 v_{t,k,j-1}+\frac 12 v_{t+1,k,j+1}
\leq 
\frac 12 v_{t,k-1,j-1}+\frac 12 v_{t+1,k-1,j+1}
=v_{2t+1,k,j}$
for $k\geq 1$.


\section{Proof of Theorem~\ref{thm:main}}
Let us define the values
\[
\gamma_{t,k,j}=\beta_{t,k,j}+\beta_{t^\comp_k,k,-j}.
\]
and
\[
\Gamma_{t,k,j}=\sum_{i\geq j}\gamma_{t,k,i}.
\]
By equation~\eqref{eqn:two_parts} the Tu--Deng Conjecture states that
$P_{t,k}
=\Gamma_{t,k,1}/2^k\leq 1/2$.

Our strategy is to show that the standard deviation of the random variable $t\mapsto \Gamma_{t,k,1}$ is much smaller than the distance to $2^{k-1}$, such that the values $P_{t,k}$ concentrate below $1/2$ by Chebyshev's inequality.
We are therefore interested in the mean value and the variance of
$t\mapsto \Gamma_{t,k,1}$ on the intervals   
$[0,2^k)$.                                   
First, we want to find a recurrence for the values
\[\beta_{t,k,j}=\bigl \lvert\bigl\{a\in\{0,\ldots,t\}: w\bigl(a+t^\comp_k\bigr)-w(a)=j\bigr \} \bigr \rvert,\]
where $k\geq 1$, $0\leq t<2^k$ and $j\in \dZ$.
For convenience, we set $\beta_{-1,j,k}=0$.
\begin{proposition}\label{prp:beta_recurrence}
Let $k\geq 0$ and $j$ be integers. Then
\begin{align*}
\beta_{0,k,j}&=\delta_{k,j},\\
\beta_{0^\comp_k,k,j}&=2^k\delta_{j,0},\\
\beta_{2t,k+1,j}&=\beta_{t,k,j-1}+
\beta_{t-1,k,j+1}&&\textrm{for }0\leq t<2^k,
\\
\beta_{2t+1,k+1,j}&=2\beta_{t,k,j}&&\textrm{for }0\leq t<2^k,\\
\beta_{(2t)^\comp_{k+1},k+1,j}&=2\beta_{t^\comp_k,k,j}&&\textrm{for }0\leq t<2^k,\\
\beta_{(2t+1)^\comp_{k+1},k+1,j}&=\beta_{t^\comp_k,k,j-1}
+\beta_{(t+1)^\comp_k,k,j+1}
&&\textrm{for }0\leq t<2^k.
\end{align*}
Furthermore, we have $\beta_{t,k,j}=0$ for $\lvert j\rvert>k$.
\end{proposition}
\begin{proof}
The last claim $\beta_{t,k,j}=0$ for $\lvert j\rvert>k$ follows by induction.
The first two statements and the cases $t=0$ are clear.
We note the almost trivial identities $(2t)^\comp_{k+1}=2t^\comp_k+1$, $(2t+1)^\comp_{k+1}=2t^\comp_k$ and $(t+1)^\comp_k=t^\comp_k-1$, which hold for all $t$ and $k$.
We calculate for $1\leq t<2^k$:
\begin{align*}
\beta_{2t,k+1,j}
&=\bigl \lvert\bigl\{a\in\{0,\ldots,2t\}: w\bigl(a+(2t)^\comp_{k+1}\bigr)-w(a)=j\bigr \rvert\\
&=\bigl \lvert\bigl\{a\in\{0,\ldots,t\}: w\bigl(2a+2t^\comp_k+1\bigr)-w(2a)=j\bigr \rvert\\
&\quad+\bigl \lvert\bigl\{a\in\{0,\ldots,t-1\}: w\bigl(2a+2t^\comp_k+2\bigr)-w(2a+1)=j\bigr \rvert\\
&=\beta_{t,k,j-1}
+\bigl \lvert\bigl\{a\in\{0,\ldots,t-1\}: w\bigl(a+(t-1)^\comp_k\bigr)-w(a)=j+1\bigr \rvert\\
&=\beta_{t,k,j-1}+
\beta_{t-1,k,j+1}.
\end{align*}
The statement also holds for $t=0$, using $\beta_{-1,k,j}=0$.
Moreover, for $0\leq t<2^k$ we have
\begin{align*}
\beta_{2t+1,k+1,j}&=\bigl \lvert\bigl\{a\in\{0,\ldots,2t+1\}: w\bigl(a+(2t+1)^\comp_{k+1}\bigr)-w(a)=j\bigr \rvert\\
&=
\bigl \lvert\bigl\{a\in\{0,\ldots,t\}: w\bigl(2a+2t^\comp_k\bigr)-w(2a)=j\bigr \rvert\\
&\quad+\bigl \lvert\bigl\{a\in\{0,\ldots,t\}: w\bigl(2a+2t^\comp_k+1\bigr)-w(2a+1)=j\bigr \rvert\\
&=2\beta_{t,k,j}
\end{align*}
and
\begin{align*}
\beta_{(2t)^\comp_{k+1},k+1,j}
&=\bigl \lvert\bigl\{a\in\{0,\ldots,2t^{\comp}_k+1\}: w(a+2t)-w(a)=j\bigr\}\bigr \rvert\\
&=\bigl \lvert\bigl\{a\in\{0,\ldots,t^\comp_k\}: w(2a+2t)-w(2a)=j\bigr\}\bigr \rvert\\
&\quad+\bigl \lvert\bigl\{a\in\{0,\ldots,t^\comp_k\}: w(2a+2t+1)-w(2a+1)=j\bigr\}\bigr \rvert
\\&=2\beta_{t^\comp_k,k,j}.
\end{align*}
Finally, for $0\leq t<2^k-1$ we have
\begin{align*}
\beta_{(2t+1)^\comp_k,k+1,j}
&=\bigl \lvert\bigl\{a\in\{0,\ldots,2t^{\comp}_k\}: w(a+2t+1)-w(a)=j\bigr\}\bigr \rvert\\
&=\bigl \lvert\bigl\{a\in\{0,\ldots,t^\comp_k\}: w(2a+2t+1)-w(2a)=j\bigr\}\bigr \rvert\\
&\quad+\bigl \lvert\bigl\{a\in\{0,\ldots,t^\comp_k-1\}: w(2a+2t+2)-w(2a+1)=j\bigr\}\bigr \rvert
\\&=\beta_{t^\comp_k,k,j-1}
+\bigl \lvert\bigl\{a\in\{0,\ldots,(t+1)^\comp_k\}: w(a+t+1)-w(a)=j+1\bigr\}\bigr \rvert
\\&=\beta_{t^\comp_k,k,j-1}
+\beta_{(t+1)^\comp_k,k,j+1}
\end{align*}
and the last statement also holds for $t=2^k-1$.
\end{proof}

We want to compute the first moments of the values $\beta_{t,k,j}$.
Define
\[    m_{k,j}   =   \sum_{t=0}^{2^k-1}\beta_{t,k,j}.    \]
Clearly, we have
\[
m_{0,j}=\delta_{0,j}.
\]
Using the above recurrence, we obtain for $k\geq 1$
\begin{align*}
m_{k,j}&=\sum_{t=0}^{2^{k-1}-1}\beta_{2t,k,j}
+\sum_{t=0}^{2^{k-1}-1}\beta_{2t+1,k,j}\\
&=
\sum_{t=0}^{2^{k-1}-1}
\bigl(\beta_{t,k-1,j-1}+\beta_{t-1,k-1,j+1}\bigr)
+2\sum_{t=0}^{2^{k-1}-1}\beta_{t,k-1,j}\\
&=
\sum_{t=0}^{2^{k-1}-1}\beta_{t,k-1,j-1}
+\sum_{t=0}^{2^{k-1}-2}\beta_{t,k-1,j+1}+2 m_{k-1,j}
\\
&= m_{k-1,j-1}+2m_{k-1,j}+m_{k-1,j+1}
-\beta_{2^{k-1}-1,k-1,j+1}
\\
&= m_{k-1,j-1}+2m_{k-1,j}+m_{k-1,j+1}
-2^{k-1}\delta_{j,-1}
\end{align*}

We define the bivariate generating function $F$:
\[
F(x,y)=\sum_{\substack{k\geq 0\\\ell\geq 0}}
m_{k,k-\ell}
x^ky^\ell.
\]
Since $\beta_{t,k,j}=0$ for $j>k$ and $0\leq t<2^k$ (which can be proved by induction) this function captures all interesting values.
Moreover, we have $\beta_{t,k,j}=0$ for $j\leq -k+1$.

Using the recurrence for $m_{k,j}$, we obtain
\begin{align*}
F(x,y)&=
\sum_{\ell\geq 0}m_{0,-\ell}y^\ell
+\sum_{\substack{k\geq 1\\\ell\geq 0}}m_{k,k-\ell}x^ky^\ell
\\&=
1+\sum_{\substack{k\geq 1\\\ell\geq 0}}x^ky^\ell
\biggl(m_{k-1,k-1-\ell}+2m_{k-1,k-\ell}+m_{k-1,k+1-\ell}
-2^{k-1}\delta_{k-\ell,-1}\biggr)
\\&=
1+xF(x,y)+2\sum_{k\geq 1}x^km_{k-1,k}+2xyF(x,y)+\sum_{\substack{k\geq 1\\0\leq \ell\leq 1}}x^ky^\ell m_{k-1,k+1-\ell}
\\&\quad+xy^2 F(x,y)- \sum_{k\geq 1}2^{k-1}x^ky^{k+1}\\
&=1+x(1+y)^2F(x,y)-\frac{xy^2}{1-2xy},
\end{align*}
therefore
\begin{align*}
F(x,y)
=
\frac{1-2xy-xy^2}{(1-2xy)\bigl(1-x(1+y)^2\bigr)}
\end{align*}

Moreover, we define
\[  \widetilde m_{k,j} \coloneqq \sum_{t=0}^{2^k-1}\beta_{t^\comp_k,k,-j} =m_{k,-j}\]
and
\[\widetilde F(x,y)\coloneqq
\sum_{\substack{k\geq 0\\\ell\geq 0}}
\widetilde m_{k,k-\ell}
x^ky^\ell.
\]
As above, we calculate for $k\geq 1$:
\begin{align*}
\widetilde m_{k,j}&=\sum_{t=0}^{2^{k-1}-1}\beta_{(2t)^\comp_k,k,-j}
+\sum_{t=0}^{2^{k-1}-1}\beta_{(2t+1)^\comp_k,k,-j}\\
&=2\sum_{t=0}^{2^{k-1}-1}\beta_{t^\comp_{k-1},k-1,-j}
+\beta_{(2^k-1)^\comp_k,k,-j}
+\sum_{t=0}^{2^{k-1}-2}\beta_{t^\comp_{k-1},k-1,-j-1}
\\&\quad+\sum_{t=0}^{2^{k-1}-2}\beta_{(t+1)^\comp_{k-1},k-1,-j+1}
\\
&=2\widetilde m_{k-1,j}+\widetilde m_{k-1,j-1}+\widetilde m_{k-1,j+1}
\\&\quad-\beta_{(2^{k-1}-1)^\comp_{k-1},k-1,-j-1}
-\beta_{0^\comp_{k-1},k-1,-j+1}+\delta_{k,-j}
\\&=\widetilde m_{k-1,j-1}+2\widetilde m_{k-1,j}+\widetilde m_{k-1,j+1}-2^{k-1}\delta_{j,1}
\end{align*}
Therefore
\begin{align*}
\widetilde F(x,y)&=
\sum_{\substack{\ell\geq 0}}\widetilde m_{0,-\ell}y^\ell
+
\sum_{\substack{\ell\geq 0\\k\geq 1}}
x^k y^\ell
\biggl(
\widetilde m_{k-1,k-\ell-1}+2\widetilde m_{k-1,k-\ell}+\widetilde m_{k-1,k-\ell+1}-2^{k-1}\delta_{k-\ell,1}
\biggr)
\\&=
1+x\widetilde F(x,y)
+2\sum_{k\geq 0}x^{k+1}\widetilde m_{k,k+1}
+2xy\widetilde F(x,y)
+\sum_{\substack{k\geq 0\\0\leq \ell\leq 1}}x^{k+1}y^\ell\widetilde m_{k,k+2-\ell}
\\&\quad+xy^2 \widetilde F(x,y)-\sum_{k\geq 0}2^{k-1}x^{k+1}y^k
\\&=1+ x(1+y)^2\widetilde F(x,y)-\frac x{1-2xy}
\end{align*}
and we get
\begin{align*}
\widetilde F(x,y)
=
\frac{1-2xy-x}{(1-2xy)\bigl(1-x(1+y)^2\bigr)}.
\end{align*}

The first moments of the random variable
$t\mapsto \beta_{t,k,j}$, where $t\in\{1,\ldots,2^k-1\}$ are contained in certain \emph{diagonals} of the bivariate rational function $F(x,y)$
(to be precise, the diagonal contains the values $m_{k,j}$,
which are first moments multiplied by $2^k$).
The moments corresponding to $j=0$ are contained in the main diagonal.

We define
\[
M_{k,l}=\sum_{t=0}^{2^k-1}\Gamma_{t,k,k-\ell}
\]
and are interested in $M_{k,k-1}$.

We have
\begin{align*}
M_{k,\ell}&=\sum_{i\geq k-\ell}
\sum_{t=0}^{2^k-1}\bigl(\beta_{t,k,i}+\beta_{t^\comp_k,k,-i}\bigr)
=\sum_{i\geq k-\ell}\bigl(m_{k,i}+\widetilde m_{k,i}\bigr)
\\&=\sum_{j=0}^{\ell}\bigl(m_{k,k-j}+\widetilde m_{k,k-j}\bigr)
=\sum_{j=0}^{\ell}\bigl[x^ky^j\bigr]\bigl(F(x,y)+\widetilde F(x,y)\bigr)
\\&=\bigl[x^ky^\ell\bigr]G(x,y),
\end{align*}
where
\[ G(x,y)=\frac{2-4xy-x-xy^2}{(1-y)(1-2xy)\bigl(1-x(1+y)^2\bigr)}. \]
The first moment of $t\mapsto 2^k\Gamma_{t,k,1}$ is therefore given by
$M_{k,k-1}=\bigl[x^ky^{k-1}\bigr]G(x,y)$.
Extracting this diagonal, we rediscover the result given in the introduction.

While this result was proved in a shorter way in the introduction, we decided to also keep this longer proof, for two reasons:
on the one hand, we have captured the first moments of all $\Gamma_{t,k,k-\ell}$ in one generating function, which better shows the underlying structure;
on the other hand, this proof is a gentle introduction to the method used for the second moment later.

\begin{proposition}\label{prp_first_moment_asymp}
We have for $k \geq 1$
\begin{align*}
	M_{k,k-1} &= \frac{1}{2} \left(4^k - {2k \choose k} \right) +1\\
	          &= \frac{4^k}{2}\left( 1 - \frac{1}{\sqrt{\pi k}} + \frac{1}{8 \sqrt{\pi k^3}} - \frac{1}{128\sqrt{\pi k^5}} + O\left(\frac{1}{\sqrt{k^7}}\right) \right).
\end{align*}
\end{proposition}

\begin{proof}
The idea of the proof is to extract the (shifted) diagonal of $G(x,y)$. First note that $[x^k y^{k-1}]G(x,y) = [x^k y^k] yG(x,y)$. The diagonal is given by $\Delta(yG)(z) := \sum_{k \geq 1} M_{k,k-1} z^k$. The computation is then a routine exercise in enumerative combinatorics (see e.g.~\cite[Chapter~6.3]{S1999})
and can be automatized to a great extent using computer algebra.
We do not present this standard argument here. 
More details can be found in the accompanying Maple Worksheet~\cite{web} implementing the manipulations on the power series using the \texttt{gfun} package~\cite{SZ1994}.

We get
\begin{align*}
	\Delta(yG)(z) &= \frac{1}{2} \left( \frac{1}{1-4z} - \frac{1}{\sqrt{1-4z}}\right)
\end{align*}
from which we extract coefficients noting $\sum_{n \geq 0} {2n \choose n} z^n = (1-4z)^{-1/2}$. The asymptotics is directly computed (to any needed order) from the known asymptotics of the central binomial coefficient.
\end{proof}

We proceed to the second moments of the values
$\Gamma_{t,k,j}$.
Define
\[M^{(2)}_{k,\ell,m}
=\sum_{0\leq t<2^k}
\Gamma_{t,k,k-\ell}\Gamma_{t,k,k-m}.
\]

The second moment of $t\mapsto \Gamma_{t,k,1}=P_{t,k}$ is obviously given by
$\frac 1{8^k}M^{(2)}_{k,k-1,k-1}$, which we want to realize as a diagonal of a trivariate rational generating function.

\begin{proposition}\label{prp:second_moment}
We have
\[
M^{(2)}_{k,\ell,m}=
\bigl[x^ky^\ell z^m\bigr]
F(x,y,z),\]
where
\[
F(x,y,z)=\frac 1{1-y}\frac 1{1-z}\bigl(A+A'+A''+A'''\bigr)(x,y,z),
\]
\begin{align*}
A(x,y,z)&=\frac{1-\frac {xy^2z^2}{1-4xyz}\Bigl(1+\frac{2xy}{1-2xy(1+yz)}+\frac{2xz}{1-2xz(1+yz)}\Bigr)}{D(x,y,z)}\\
A'(x,y,z)&=
\frac 1{1-2xz(1+yz)}\cdot
\frac{1-x(1+yz)^2-\frac{xyz}{1-2xy(1+yz)}-xyz}{D(x,y,z)}\\
A''(x,y,z)&=
\frac 1{1-2xy(1+yz)}\cdot
\frac{1-x(1+yz)^2-\frac{xyz}{1-2xz(1+yz)}-xyz}{D(x,y,z)}\\
A'''(x,y,z)&=\frac{1-\frac{x}{1-4xyz}\Bigl(1+\frac{2xy^2z}{1-2xy(1+yz)}+\frac{2xyz^2}{1-2xz(1+yz)}\Bigr)}{D(x,y,z)}
\end{align*}
and
\[D(x,y,z)=1-x(1+yz)^2-\frac{xyz}{1-2xy(1+yz)}-\frac{xyz}{1-2xz(1+yz)}.\]
\end{proposition}
\begin{proposition}\label{prp_diagonal}
We have the asymptotic expansion
\[\frac 1{8^k}M^{(2)}_{k,k-1,k-1}=\frac 14 - \frac 1{2\sqrt{\pi k}} + \frac 1{4\pi k}+\frac 1{16\sqrt{\pi}k^{3/2}}+\frac{17}{72\pi k^2}+\LandauO(k^{-5/2}).
\]
\end{proposition}

\begin{corollary}
Let $X_k$ be the discrete random variable defined by $X_k(t)=P_{t,k}=\lvert S_{t,k}\rvert /2^k$, where $1\leq t<2^k-1$, and let
$\sigma_k=\sqrt{\mathbb E(X_k-\mathbb E X_k)^2}$ be the corresponding standard deviation.
Then for $k\rightarrow\infty$ we have
\[\sigma_k\sim\frac{\sqrt{43}}{12\sqrt{\pi}}k^{-1}.\]
\end{corollary}

\begin{proof}
The first and second moments of the random variable
$t\mapsto \frac 1{2^k}\lvert S_{t,k}\rvert$ are given by
$\frac 1{4^k}M_{k,k-1}$ and $\frac 1{8^k}M^{(2)}_{k,k-1,k-1}$,
of which the asymptotics have been computed in
Propositions~\ref{prp_first_moment_asymp} and~\ref{prp_diagonal}.
By considering $\mathbb E(X_k^2)-(\mathbb EX_k)^2$, we see that all terms up to $\LandauO(k^{-2})$ cancel, leaving only the asymptotics $43/(144\pi k^2)+\LandauO(k^{-5/2})$.

\end{proof}

Finally, in an analogous manner as in~\cite[Section~4.4]{DKS2016} the proof of Theorem~\ref{thm:main} is completed by Chebyshev's inequality.

The remaining part of this paper is devoted to the proofs of Propositions~\ref{prp:second_moment} and \ref{prp_diagonal}.
\subsection{Proof of Proposition~\ref{prp:second_moment}}
\begin{proof}[Proof of Proposition~\ref{prp:second_moment}]
Define
\[
a_{k,\ell,m}=
\sum_{t=0}^{2^k-1}
\beta_{t,k,k-\ell}\beta_{t,k,k-m}.
\]
and auxiliary values

\begin{align*}
b_{k,\ell,m}&=
\sum_{t=0}^{2^k-2}\beta_{t,k,k-\ell}\beta_{t+1,k,k-m},\\
c_{k,\ell,m}&=
\sum_{t=0}^{2^k-2}\beta_{t+1,k,k-\ell}\beta_{t,k,k-m}.
\end{align*}

We calculate, for $k\geq 1$ and $\ell,m\geq 0$:
\begin{align*}
a_{k,\ell,m}&=
\sum_{t=0}^{2^{k-1}-1}
\beta_{2t,k,k-\ell}\beta_{2t,k,k-m}
+
\sum_{t=0}^{2^{k-1}-1}
\beta_{2t+1,k,k-\ell}\beta_{2t+1,k,k-m}
\\&=
\sum_{0\leq t<2^{k-1}}\bigl(\beta_{t,k-1,k-1-\ell}+\beta_{t-1,k-1,k+1-\ell}\bigr)
\bigl(\beta_{t,k-1,k-1-m}+\beta_{t-1,k-1,k+1-m}\bigr)
\\&\quad+4\sum_{0\leq t<2^{k-1}}\beta_{t,k-1,k-\ell}\beta_{t,k-1,k-m}
\\&=
\sum_{0\leq t<2^{k-1}}\beta_{t,k-1,k-1-\ell}\beta_{t,k-1,k-1-m}
+\sum_{0\leq t<2^{k-1}-1}\beta_{t,k-1,k+1-\ell}\beta_{t+1,k-1,k-1-m}
\\&\quad+\sum_{0\leq t<2^{k-1}-1}\beta_{t+1,k-1,k-1-\ell}\beta_{t,k-1,k+1-m}
+\sum_{0\leq t<2^{k-1}-1}\beta_{t,k-1,k+1-\ell}\beta_{t,k-1,k+1-m}
\\&\quad+4\sum_{0\leq t<2^{k-1}}\beta_{t,k-1,k-\ell}\beta_{t,k-1,k-m}
\\&=a_{k-1,\ell,m}+b_{k-1,\ell-2,m}+c_{k-1,\ell,m-2}
+a_{k-1,\ell-2,m-2}+4a_{k-1,\ell-1,m-1}
\\&\quad-\beta_{2^{k-1}-1,k-1,k+1-\ell}\beta_{2^{k-1}-1,k-1,k+1-m}
\\&=a_{k-1,\ell,m}+b_{k-1,\ell-2,m}+c_{k-1,\ell,m-2}
+a_{k-1,\ell-2,m-2}+4a_{k-1,\ell-1,m-1}
\\&\quad-2^{2(k-1)}\delta_{k+1,\ell}\delta_{k+1,m}
\end{align*}
Assume now that $k\geq 1$. We have
\begin{align*}
b_{k,\ell,m}&=\sum_{0\leq t<2^{k-1}}\beta_{2t,k,k-\ell}\beta_{2t+1,k,k-m}
+\sum_{0\leq t<2^{k-1}-1}\beta_{2t+1,k,k-\ell}\beta_{2t+2,k,k-m}
\\&=
\sum_{0\leq t<2^{k-1}}\bigl(\beta_{t,k-1,k-1-\ell}+\beta_{t-1,k-1,k+1-\ell}\bigr)2\beta_{t,k-1,k-m}
\\&\quad+\sum_{0\leq t<2^{k-1}-1}2\beta_{t,k-1,k-\ell}\bigl(\beta_{t+1,k-1,k-1-m}+\beta_{t,k-1,k+1-m}\bigr)
\end{align*}
Noting that $\beta_{1,k,k-m}=2\beta_{0,k-1,k-m}$, we obtain
\begin{align*}
b_{k,\ell,m}&=
2\sum_{0\leq t<2^{k-1}}\beta_{t,k-1,k-1-\ell}\beta_{t,k-1,k-m}
\\&\quad+2\sum_{0\leq t<2^{k-1}-1}\beta_{t,k-1,k+1-\ell}\beta_{t+1,k-1,k-m}
+2\sum_{0\leq t<2^{k-1}-1}\beta_{t,k-1,k-\ell}\beta_{t+1,k-1,k-1-m}
\\&\quad+2\sum_{0\leq t<2^{k-1}}\beta_{t,k-1,k-\ell}\beta_{t,k-1,k+1-m}
-2\beta_{2^{k-1}-1,k-1,k-\ell}\beta_{2^{k-1}-1,k-1,k+1-m}
\\&=2a_{k-1,\ell,m-1}+2b_{k-1,\ell-2,m-1}+2b_{k-1,\ell-1,m}+2a_{k-1,\ell-1,m-2}
-2^{2k-1}\delta_{k,\ell}\delta_{k+1,m}.
\end{align*}
By the obvious identities $a_{k,\ell,m}=a_{k,m,\ell}$ and
$b_{k,\ell,m}=c_{k,m,\ell}$ we have
\[c_{k,\ell,m}=
2a_{k-1,\ell-1,m}+2c_{k-1,\ell-1,m-2}+2c_{k-1,\ell,m-1}+2a_{k-1,\ell-2,m-1}
-2^{2k-1}\delta_{k+1,\ell}\delta_{k,m}.
\]

We define generating functions
\begin{align*}
A(x,y,z)&=\sum_{k,\ell,m\geq 0}a_{k,\ell,m}x^ky^\ell z^m\\
B(x,y,z)&=\sum_{k,\ell,m\geq 0}b_{k,\ell,m}x^ky^\ell z^m\\
C(x,y,z)&=\sum_{k,\ell,m\geq 0}c_{k,\ell,m}x^ky^\ell z^m\\
\end{align*}
Summing over $k,\ell,m$, the above recurrences translates to identities for these functions:
noting that $a_{k,\ell,m}=0$ for $\ell<0$ or $m<0$, and that
\begin{align*}\sum_{\ell,m\geq 0}a_{0,\ell,m}y^\ell z^m
&=
\sum_{\ell,m\geq 0}\beta_{0,0,-\ell}\beta_{0,0,-m}y^\ell z^m
=1,
\end{align*}
we obtain
\begin{align*}
A(x,y,z)&=1+x(1+4yz+y^2z^2)A(x,y,z)+xy^2B(x,y,z)+xz^2C(x,y,z)
\\&\quad-\frac 14\sum_{k\geq 1}4^k x^ky^{k+1}z^{k+1}\\
&=1+x(1+4yz+y^2z^2)A(x,y,z)+xy^2B(x,y,z)+xz^2C(x,y,z)
\\&\quad-\frac {yz}4\frac {4xyz}{1-4xyz}.
\end{align*}
Moreover, we have $\sum_{\ell,m\geq 0}b_{0,\ell,m}y^\ell z^m=0$, therefore
\begin{align*}
B(x,y,z)&=2xz(1+yz)A(x,y,z)+2xy(1+yz)B(x,y,z)
-\frac 12\sum_{k\geq 1}4^kx^ky^kz^{k+1}
\\&=2xz(1+yz)A(x,y,z)+2xy(1+yz)B(x,y,z)
-\frac z2\frac{4xyz}{1-4xyz}.
\end{align*}
Finally, we have
\begin{align*}
C(x,y,z)&=2xy(1+yz)A(x,y,z)+2xz(1+yz)C(x,y,z)
-\frac 12\sum_{k\geq 1}4^kx^ky^{k+1}z^k
\\&=2xy(1+yz)A(x,y,z)+2xz(1+yz)C(x,y,z)
-\frac y2\frac{4xyz}{1-4xyz}.
\end{align*}

We have
\[B(x,y,z)=\frac{2xz(1+yz)A(x,y,z)-\frac z2\frac{4xyz}{1-4xyz}}{1-2xy(1+yz)}\]
and
\[C(x,y,z)=\frac{2xy(1+yz)A(x,y,z)-\frac y2\frac{4xyz}{1-4xyz}}{1-2xz(1+yz)}.\]
Inserting these identities into the equation for $A(x,y,z)$, we obtain
\begin{multline*}
A(x,y,z)\biggl(1-x(1+4yz+y^2z^2)-xy^2\frac{2xz(1+yz)}{1-2xy(1+yz)}-xz^2\frac{2xy(1+yz)}{1-2xz(1+yz)}
\biggr)\\
=1-\frac{yz}4\frac{4xyz}{1-4xyz}-xy^2\frac{\frac z2\frac{4xyz}{1-4xyz}}{1-2xy(1+yz)}-xz^2\frac{\frac y2\frac{4xyz}{1-4xyz}}{1-2xz(1+yz)}
\end{multline*}
After some rewriting we obtain
\[
A(x,y,z)=\frac{1-\frac {xy^2z^2}{1-4xyz}\Bigl(1+\frac{2xy}{1-2xy(1+yz)}+\frac{2xz}{1-2xz(1+yz)}\Bigr)}{1-x(1+yz)^2-\frac{xyz}{1-2xy(1+yz)}-\frac{xyz}{1-2xz(1+yz)}}
\]
Note that the denominator is the same as in~\cite[Equation~(19)]{DKS2016}.

Define
\begin{align*}
a'_{k,\ell,m}&=\sum_{0\leq t<2^k}\beta_{t,k,k-\ell}\beta_{t^\comp_k,k,-k+m}\\
b'_{k,\ell,m}&=\sum_{0\leq t<2^k-1}\beta_{t,k,k-\ell}\beta_{(t+1)^\comp_k,k,-k+m}\\
c'_{k,\ell,m}&=\sum_{0\leq t<2^k-1}\beta_{t-1,k,k-\ell}\beta_{(t+1)^\comp_k,k,-k+m}\\
\end{align*}

We have for $k\geq 1$
\begin{align*}
a'_{k,\ell,m}&=
\sum_{0\leq t<2^{k-1}}\beta_{2t,k,k-\ell}\beta_{(2t)^\comp_k,k,-k+m}
+
\sum_{0\leq t<2^{k-1}}\beta_{2t+1,k,k-\ell}\beta_{(2t+1)^\comp_k,k,-k+m}
\\&=\sum_{0\leq t<2^{k-1}}\bigl(\beta_{t,k-1,k-\ell-1}+\beta_{t-1,k-1,k-\ell+1}\bigr)2\beta_{t^\comp_{k-1},k-1,-k+m}
\\&\quad+\sum_{0\leq t<2^{k-1}}2\beta_{t,k-1,k-\ell}\bigl(\beta_{t^\comp_{k-1},k-1,-k+m-1}+\beta_{(t+1)^\comp_{k-1},k-1,-k+m+1}\bigr)
\\&=2\sum_{0\leq t<2^{k-1}}\beta_{t,k-1,k-1-\ell}\beta_{t^\comp_{k-1},k-1,-(k-1)+m-1}
\\&\quad+2\sum_{0\leq t<2^{k-1}-1}\beta_{t,k-1,k-1-(\ell-2)}\beta_{(t+1)^\comp_{k-1},k-1,-(k-1)+m-1}
\\&\quad+2\sum_{0\leq t<2^{k-1}}\beta_{t,k-1,k-1-(\ell-1)}\beta_{t^\comp_{k-1},k-1,-(k-1)+m-2}
\\&\quad+2\sum_{0\leq t<2^{k-1}}\beta_{t,k-1,k-1-(\ell-1)}\beta_{(t+1)^\comp_{k-1},k-1,-(k-1)+m}
\\&=2a'_{k-1,\ell,m-1}
+2b'_{k-1,\ell-2,m-1}
+2a'_{k-1,\ell-1,m-2}
+2b'_{k-1,\ell-1,m}
\end{align*}
Moreover
\begin{align*}
b'_{k,\ell,m}
&=\sum_{0\leq t<2^{k-1}}\beta_{2t,k,k-\ell}\beta_{(2t+1)^\comp_k,k,-k+m}
\\&\quad+\sum_{0\leq t<2^{k-1}-1}\beta_{2t+1,k,k-\ell}\beta_{(2(t+1))^\comp_k,k,-k+m}
\\&=\sum_{0\leq t<2^{k-1}}
\bigl(\beta_{t,k-1,k-\ell-1}+\beta_{t-1,k-1,k-\ell+1}\bigr)
\bigl(\beta_{t^\comp_{k-1},k-1,-k+m-1}+\beta_{(t+1)^\comp_{k-1},k-1,-k+m+1}\bigr)
\\&\quad+4\sum_{0\leq t<2^{k-1}-1}\beta_{t,k-1,k-\ell}\beta_{(t+1)^\comp_{k-1},k-1,-k+m}
\\&=a'_{k-1,\ell,m-2}+b'_{k-1,\ell-2,m-2}+b'_{k-1,\ell,m}+c'_{k-1,\ell-2,m}+4b'_{k-1,\ell-1,m-1}
\end{align*}
and
\begin{align*}
c'_{k,\ell,m}&=\sum_{0\leq t<2^{k-1}}\beta_{2t-1,k,k-\ell}\beta_{(2t+1)^\comp_k,k,-k+m}
+\sum_{0\leq t<2^{k-1}}\beta_{2t,k,k-\ell}\beta_{(2(t+1))^\comp_k,k,-k+m}
\\&=2\sum_{0\leq t<2^{k-1}}\beta_{t-1,k-1,k-\ell}\bigl(\beta_{t^\comp_{k-1},k-1,-k+m-1}+\beta_{(t+1)^\comp_{k-1},k-1,-k+m+1}\bigr)
\\&\quad+2\sum_{0\leq t<2^{k-1}}\bigl(\beta_{t,k-1,k-\ell-1}+\beta_{t-1,k-1,k-\ell+1}\bigr)\beta_{(t+1)^\comp_{k-1},k-1,-k+m}
\\&=2b_{k-1,\ell-1,m-2}+2c_{k-1,\ell-1,m}+2b_{k-1,\ell,m-1}+2c_{k-1,\ell-2,m-1}
\end{align*}

We define generating functions
\begin{align*}
A'(x,y,z)&=\sum_{k,\ell,m\geq 0}a'_{k,\ell,m}x^ky^\ell z^m\\
B'(x,y,z)&=\sum_{k,\ell,m\geq 0}b'_{k,\ell,m}x^ky^\ell z^m\\
C'(x,y,z)&=\sum_{k,\ell,m\geq 0}c'_{k,\ell,m}x^ky^\ell z^m.
\end{align*}

We have $a_{k,\ell,m}=0$ for $\ell<0$ or $m<0$, and
\begin{align*}\sum_{\ell,m\geq 0}a'_{0,\ell,m}y^\ell z^m
&=
\sum_{\ell,m\geq 0}\beta_{0,0,-\ell}\beta_{0,0,-m}y^\ell z^m
=1,
\end{align*}
moreover 
\begin{align*}\sum_{\ell,m\geq 0}b'_{0,\ell,m}y^\ell z^m
&=
\sum_{\ell,m\geq 0}c'_{0,\ell,m}y^\ell z^m
=0.
\end{align*}

We obtain
\begin{equation}\label{eqn_A_B}
\begin{aligned}
A'(x,y,z)&=1+2xzA'(x,y,z)+2xy^2zB'(x,y,z)+2xyz^2A'(x,y,z)+2xyB'(x,y,z)\\
&=
1+2xz(1+yz)A'(x,y,z)+2xy(1+yz)B'(x,y,z),
\end{aligned}
\end{equation}
\begin{align*}
B'(x,y,z)&=xz^2A'(x,y,z)+(xy^2z^2+x+4xyz)B'(x,y,z)+xy^2C'(x,y,z)
\end{align*}
and
\begin{align*}
C'(x,y,z)&=2xz(1+yz)B'(x,y,z)+2xy(1+yz)C'(x,y,z).
\end{align*}

It follows that
\[A'(x,y,z)=\frac{1+2xy(1+yz)B'(x,y,z)}{1-2xz(1+yz)}\]
\[C'(x,y,z)=\frac{2xz(1+yz)B'(x,y,z)}{1-2xy(1+yz)}\]
and therefore
\begin{align*}
B'(x,y,z)&=
\biggl(xz^2\frac{2xy(1+yz)}{1-2xz(1+yz)}+x(1+4yz+y^2z^2)+xy^2\frac{2xz(1+yz)}{1-2xy(1+yz)}\biggr)B'(x,y,z)
\\&+\frac {xz^2}{1-2xz(1+yz)}.
\end{align*}
Inserting this into~\eqref{eqn_A_B}, we obtain after some elementary manipulation
\[
A'(x,y,z)=
\frac 1{1-2xz(1+yz)}\cdot
\frac{1-x(1+yz)^2-\frac{xyz}{1-2xy(1+yz)}-xyz}{1-x(1+yz)^2-\frac{xyz}{1-2xy(1+yz)}-\frac{xyz}{1-2xz(1+yz)}}.
\]


Define
\[
a''_{k,\ell,m}=\sum_{0\leq t<2^k}\beta_{t^\comp_k,k,-k+\ell}\beta_{t,k,k-m}
\]
and
\[A''(x,y,z)=\sum_{k,\ell,m\geq 0}a''_{k,\ell,m}x^ky^\ell z^m
\]
By exchanging the roles of $\ell$ and $m$ resp. $y$ and $z$ we obtain
\[
A''(x,y,z)=
\frac 1{1-2xy(1+yz)}\cdot
\frac{1-x(1+yz)^2-\frac{xyz}{1-2xz(1+yz)}-xyz}{1-x(1+yz)^2-\frac{xyz}{1-2xy(1+yz)}-\frac{xyz}{1-2xz(1+yz)}}.
\]

Finally, we define

\begin{align*}
a'''_{k,\ell,m}&=\sum_{0\leq t<2^k}\beta_{t^\comp_k,k,-k+\ell}\beta_{t^\comp_k,k,-k+m}\\
b'''_{k,\ell,m}&=\sum_{0\leq t<2^k}\beta_{t^\comp_k,k,-k+\ell}\beta_{(t+1)^\comp_k,k,-k+m}\\
c'''_{k,\ell,m}&=\sum_{0\leq t<2^k}\beta_{(t+1)^\comp_k,k,-k+\ell}\beta_{t^\comp_k,k,-k+m}
\end{align*}
and we have
\begin{align*}
a'''_{k,\ell,m}&=
\sum_{0\leq t<2^{k-1}}\beta_{(2t)^\comp_k,k,-k+\ell}\beta_{(2t)^\comp_k,k,-k+m}
+
\sum_{0\leq t<2^{k-1}}\beta_{(2t+1)^\comp_k,k,-k+\ell}\beta_{(2t+1)^\comp_k,k,-k+m}
\\&=4\sum_{0\leq t<2^{k-1}}\beta_{t^\comp_{k-1},k-1,-(k-1)+\ell-1}\beta_{t^\comp_{k-1},k-1,-(k-1)+m-1}
\\&\quad+\sum_{0\leq t<2^{k-1}}\bigl(\beta_{t^\comp_{k-1},k-1,-(k-1)+\ell-2}+\beta_{(t+1)^\comp_{k-1},k-1,-(k-1)+\ell}\bigr)
\\&\qquad\times\bigl(\beta_{t^\comp_{k-1},k-1,-(k-1)+m-2}+\beta_{(t+1)^\comp_{k-1},k-1,-(k-1)+m}\bigr)
\\&=4a'''_{k-1,\ell-1,m-1}+a'''_{k-1,\ell-2,m-2}+b'''_{k-1,\ell-2,m}+c'''_{k-1,\ell,m-2}+a'''_{k-1,\ell,m}
\\&\quad-\beta_{0^\comp_{k-1},k-1,-(k-1)+\ell}\beta_{0^\comp_{k-1},k-1,-(k-1)+m}
\\&=4a'''_{k-1,\ell-1,m-1}+a'''_{k-1,\ell-2,m-2}+b'''_{k-1,\ell-2,m}+c'''_{k-1,\ell,m-2}+a'''_{k-1,\ell,m}
\\&\quad-2^{2(k-1)}\delta_{k,\ell+1}\delta_{k,m+1}
\end{align*}

\begin{align*}
b'''_{k,\ell,m}&=
\sum_{0\leq t<2^{k-1}}\beta_{(2t)^\comp_k,k,-k+\ell}\beta_{(2t+1)^\comp_k,k,-k+m}
+
\sum_{0\leq t<2^{k-1}}\beta_{(2t+1)^\comp_k,k,-k+\ell}\beta_{(2(t+1))^\comp_k,k,-k+m}
\\&=
2\sum_{0\leq t<2^{k-1}}\beta_{t^\comp_{k-1},k-1,-(k-1)+\ell-1}\bigl(\beta_{t^\comp_{k-1},k-1,-(k-1)+m-2}+\beta_{(t+1)^\comp_{k-1},k-1,-(k-1)+m}\bigr)
\\&\quad+2\sum_{0\leq t<2^{k-1}}\bigl(\beta_{t^\comp_{k-1},k-1,-(k-1)+\ell-2}+\beta_{(t+1)^\comp_{k-1},k-1,-(k-1)+\ell}\bigr)\beta_{(t+1)^\comp_{k-1},k-1,-(k-1)+m-1}
\\&=2a'''_{k-1,\ell-1,m-2}+2b'''_{k-1,\ell-1,m}+2b'''_{k-1,\ell-2,m-1}+2a'''_{k-1,\ell,m-1}
\\&\quad-2\beta_{0^\comp_{k-1},k-1,-k+\ell+1}\beta_{0^\comp_{k-1},k-1,-k+m}
\\&=2a'''_{k-1,\ell-1,m-2}+2b'''_{k-1,\ell-1,m}+2b'''_{k-1,\ell-2,m-1}+2a'''_{k-1\ell,m-1}-2^{2k-1}\delta_{k,\ell+1}\delta_{k,m}
\end{align*}
and

\begin{align*}
c'''_{k,\ell,m}&=b'''_{k,m,\ell}
\\&=2a'''_{k-1,m-1,\ell-2}+2b'''_{k-1,m-1,\ell}+2b'''_{k-1,m-2,\ell-1}+2a'''_{k-1,m,\ell-1}-2^{2k-1}\delta_{k,m+1}\delta_{k,\ell}
\\&=2a'''_{k-1,\ell-2,m-1}+2c'''_{k-1,\ell,m-1}+2c'''_{k-1,\ell-1,m-2}+2a'''_{k-1,\ell-1,m}-2^{2k-1}\delta_{k,\ell}\delta_{k,m+1}.
\end{align*}
Again we translate this to generating functions.
We note that
\[\sum_{\ell,m\geq 0}a'''_{0,\ell,m}y^\ell z^m=
\sum_{\ell,m\geq 0}\beta_{0,0,-\ell}\beta_{0,0,m}=
\sum_{\ell,m\geq 0}\delta_{\ell,0}\delta_{m,0}=1\]
and that
\[\sum_{\ell,m\geq 0}b'''_{0,\ell,m}y^\ell z^m=
\sum_{\ell,m\geq 0}c'''_{0,\ell,m}y^\ell z^m=0.\]
Therefore
\begin{align*}A'''(x,y,z)&=
1+x(4yz+y^2z^2+1)A'''(x,y,z)+xy^2B(x,y,z)+xz^2C(x,y,z)
\\&\quad-\sum_{k\geq 1}4^{k-1}x^ky^{k-1}z^{k-1}
\\&=
1+x(4yz+y^2z^2+1)A'''(x,y,z)+xy^2B(x,y,z)+xz^2C(x,y,z)
\\&\quad-\frac x{1-4xyz}
\end{align*}
and
\begin{align*}
B'''(x,y,z)&=2xz(1+yz)A'''(x,y,z)+2xy(1+yz)B'''(x,y,z)-
2\sum_{k\geq 1}4^{k-1}x^ky^{k-1}z^k
\\&=
2xz(1+yz)A'''(x,y,z)+2xy(1+yz)B'''(x,y,z)-\frac{2xz}{1-4xyz}
\end{align*}
\[C'''(x,y,z)=2xy(1+yz)A'''(x,y,z)+2xz(1+yz)C'''(x,y,z)-\frac{2xy}{1-4xyz}.\]

It follows that
\[B'''(x,y,z)=\frac{2xz(1+yz)A'''(x,y,z)-\frac{2xz}{1-4xyz}}{1-2xy(1+yz)}\]
and
\[C'''(x,y,z)=\frac{2xy(1+yz)A'''(x,y,z)-\frac{2xy}{1-4xyz}}{1-2xz(1+yz)}\]
and therefore
\begin{align*}
A'''(x,y,z)&\biggl(1-x(1+4yz+y^2z^2)-xy^2\frac{2xz(1+yz)}{1-2xy(1+yz)}-xz^2\frac{2xy(1+yz)}{1-2xz(1+yz)}\biggr)
\\&=1-\frac{x}{1-4xyz}-xy^2\frac{\frac{2xz}{1-4xyz}}{1-2xy(1+yz)}-xz^2\frac{\frac{2xy}{1-4xyz}}{1-2xz(1+yz)}.
\end{align*}
It follows that
\[
A'''(x,y,z)=\frac{1-\frac{x}{1-4xyz}\Bigl(1+\frac{2xy^2z}{1-2xy(1+yz)}+\frac{2xyz^2}{1-2xz(1+yz)}\Bigr)}{1-x(1+yz)^2-\frac{xyz}{1-2xy(1+yz)}-\frac{xyz}{1-2xz(1+yz)}}.
\]

We have
\begin{align*}
M^{(2)}_{k,\ell,m}&=\sum_{\substack{i\leq \ell\\j\leq m}}\sum_{0\leq t<2^k}
\gamma_{t,k,k-i}\gamma_{t,k,k-j}\\
&=\sum_{\substack{i\leq \ell\\j\leq m}}\bigl(a_{k,i,j}+a'_{k,i,j}+a''_{k,i,j}+a'''_{k,i,j}\bigr)\\
&=\bigl[x^ky^\ell z^m\bigr]\frac 1{1-y}\frac 1{1-z}\bigl(A+A'+A''+A'''\bigr)(x,y,z).
\qedhere
\end{align*}
\end{proof}
\subsection{Proof of Proposition~\ref{prp_diagonal}}
We write
\[F(x,y,z)=\frac 1{(1-y)(1-z)}\frac{G(x,y,z)}{D(x,y,z)},\]
where
\[G=B+B'+B''+B''',\]
\begin{align*}B(x,y,z)&=1-\frac {xy^2z^2}{1-4xyz}\left(1+\frac{2xy}{1-2xy(1+yz)}+\frac{2xz}{1-2xz(1+yz)}\right)\\
B'(x,y,z)&=
\frac {1-x(1+yz)^2-\frac{xyz}{1-2xy(1+yz)}-xyz}{1-2xz(1+yz)}\\
B''(x,y,z)&=
\frac {1-x(1+yz)^2-\frac{xyz}{1-2xz(1+yz)}-xyz}{1-2xy(1+yz)}\\
B'''(x,y,z)&=1-\frac{x}{1-4xyz}\left(1+\frac{2xy^2z}{1-2xy(1+yz)}+\frac{2xyz^2}{1-2xz(1+yz)}\right)
\end{align*}
and
\[D(x,y,z)=1-x(1+yz)^2-\frac{xyz}{1-2xy(1+yz)}-\frac{xyz}{1-2xz(1+yz)}.\]
The proof is analogous to~\cite[Proposition~10]{DKS2016}.
We copy the following two lemmata.
(We denote the open disk with radius $\delta$ around $a\in\dC$ by $B_\delta(a)$.)
\begin{lemma}\label{lem:xexp}
There exist $\delta ,\delta_1,\varepsilon>0 $ and a unique smooth function $f:B_\delta(1)\times B_\delta(1)\rightarrow\dC$ such that $f(1,1) = 1/8$ and
\[
D(f(y,z),y,z) = 0
\]
for $|y-1|<\delta$ and $|z-1|<\delta$, such that
\begin{equation}\label{eqxexp}
[x^n]\, F(x,y,z) = \frac 1{(1-y)(1-z)}\left( \frac{-G(f(y,z),y,z)}{D_x(f(y,z),y,z)} f(y,z)^{-n-1} + O\bigl(8^{(1-\varepsilon)n}\bigr) \right)
\end{equation}
uniformly for $|y-1|<\delta$ and $|z-1|<\delta$, and such that 
\begin{equation}\label{eqxexp2}
[x^n]\, F(x,y,z) = \LandauO(8^{(1-\varepsilon)n}) 
\end{equation}
uniformly for all $y,z$ satisfying
$|y| \leq 1 + \delta_1$, $|z| \leq 1 + \delta_1$ and 
$(|y-1| \geq \delta$ or $|z-1| \ge \delta)$.
Furthermore, we have the local expansions
\begin{align*}
f(y,z) &= \frac 18-\frac 18(y-1)-\frac 18(z-1)+\frac 3{32}(y-1)^2+\frac 3{32}(z-1)^2+\frac 18(y-1)(z-1)\\
&\quad-\frac 1{16}(y-1)^3-\frac 1{16}(z-1)^3-\frac 3{32}(y-1)^2(z-1)-\frac 3{32}(y-1)(z-1)^2\\
&\quad+\frac 5{128}(y-1)^4+\frac 5{128}(z-1)^4+\frac 1{16}(y-1)^3(z-1)+\frac 1{16}(y-1)(z-1)^3\\
&\quad+\frac {13}{192}(y-1)^2(z-1)^2+O\bigl(\abs{y-1}^5+\abs{z-1}^5\bigr)
\end{align*}
and
\begin{align*}
\log f(y,z) &= -\log 8 - (y-1)-(z-1) + \frac 14(y-1)^2 + \frac 14(z-1)^2 \\
&\quad-\frac 1{12}(y-1)^3 - \frac 1{12}(z-1)^3
+\frac 1{32}(y-1)^4+\frac 1{32}(z-1)^4\\
&\quad-\frac 1{48}(y-1)^2(z-1)^2
+ O\bigl(\abs{y-1}^5+\abs{z-1}^5\bigr)
\end{align*}
 at $(1,1)\in\dC^2$.
\end{lemma}

The next lemma will be needed for computing the asymptotic expansion of the coefficients of $y^n z^n$. It summarizes results on the normal distribution.
\begin{lemma}\label{lem:integral}
We have 
\[
\int_{-\infty, \Im(s)> 0}^\infty e^{-s^2/4} \frac{\mathrm d s}s = - \pi i,
\]
and for $k\geq 0$
\[
\int_{-\infty}^\infty e^{-s^2/4}s^k\mathrm d s
=\begin{cases}
2\sqrt{\pi}\frac{k!}{(k/2)!} ,&k \mbox{ even},\\
0,&k \mbox{ odd.}
\end{cases}
\]
\end{lemma}

\newcommand*{\cosc}{A_n}

We begin by determining the coefficient $[y^{n-1} z^{n-1}]$ using Cauchy integration,
\[
[x^n y^{n-1} z^{n-1}]\, F(x,y,z) = \frac 1{(2\pi i)^2} \iint\limits_{\gamma\times \gamma} 
[x^n]\, F(x,y,z) \frac{\mathrm dy}{y^n} \frac{\mathrm dz}{z^n},
\]	
where the contour of integration $\gamma$ consists of two pieces:
a part $\gamma_1$ inside the disk of radius $\delta$ around $1$, which connects the points $1 \pm i\delta$ and passes $1$ on the left hand side, and a part $\gamma_2$, which is just a circular arc around $0$ connecting the points $1\pm i\delta$, see Figure~\ref{fig:intpath}.

By~\eqref{eqxexp} and~\eqref{eqxexp2} the integral along $\gamma_2$, is of order $\LandauO(8^{(1-\varepsilon)n})$ which will turn out to be exponentially smaller than the main part arising from the integral along $\gamma_1$.
Therefore we may replace $\gamma$ by $\gamma_1$, obtaining
\begin{multline*}
[x^n y^{n-1} z^{n-1}]\, F(x,y,z)
= O\bigl(8^{(1-\varepsilon)n}\bigr)\\
+ \frac 1{(2\pi i)^2} \iint\limits_{\gamma_1\times \gamma_1}
\frac 1{(1-y)(1-z)} \frac{-yz\,G(f(y,z),y,z)}{D_x(f(y,z),y,z)} \bigl(f(y,z)yz\bigr)^{-n-1} \mathrm dy\,\mathrm dz.
\end{multline*}
For $y,z\in \gamma_1$ we set 
\[
y = 1 + i \frac s{\sqrt n} \quad \mbox{and}\quad z = 1 + i \frac t{\sqrt n}
\]
and obtain after this substitution
\[
[x^n y^{n-1} z^{n-1}]\, F(x,y,z) = \frac 1{(2\pi i)^2} \iint\limits_{\substack{|s|,|t|\le \delta \sqrt n,\\ \Im(s),\Im(t)> 0}} 
P_n(s,t) e^{-(n+1)\, g_n(s,t)} \frac{\mathrm ds\,\mathrm dt}{st} + O\bigl(8^{(1-\varepsilon)n}\bigr),
\]
where 
\[
P_n(s,t) = \left.\frac{-yz\,G(f(y,z),y,z)}{D_x(f(y,z),y,z)}\right|_{y = 1 + i s/{\sqrt n},\, z = 1 + i t/{\sqrt n} }
\]
and
\[
g_n(s,t) = \left.\left( \log f(y,z) + \log y + \log z \right)\right|_{y = 1 + i s/{\sqrt n},\, z = 1 + i t/{\sqrt n} }.
\]
Using the Taylor expansion of $f(x,y)$ and a computer algebra system, we obtain
\begin{align*}
\frac{-yz\,G(f(y,z),y,z)}{D_x(f(y,z),y,z)} &=
\frac 18 - \frac 1{32}(y-1)^2 - \frac 1{32}(z-1)^2+O\bigl(\abs{y-1}^3+\abs{z-1}^3\bigr),
\end{align*}
from which it follows that
\[
P_n(s,t) = \frac 18\left(1 +\frac{s^2}{4n}+\frac{t^2}{4n}
+O\left( \frac{|s|^3+|t|^3}{n^{3/2}} \right)\right).
\]

Lemma~\ref{lem:xexp} implies
\begin{align*}
\log f(y,z) + \log y + \log z &= - \log 8 - \frac 14 (y-1)^2 - \frac 14 (z-1)^2+ \frac 1{4} (y-1)^3 + \frac 1{4} (z-1)^3 \\
&\quad-\frac 7{32}(y-1)^4-\frac 7{32}(y-1)^4-\frac 1{48}(y-1)^2(z-1)^2\\
&\quad+ O\bigl(|y-1|^5+|z-1|^5\bigr), 
\end{align*}
so that 
\begin{align}
\label{eq:ng}
\begin{aligned}
-(n+1)\, g_n(s,t)
&= \log 8^{n+1} - \frac{s^2}4 - \frac{t^2}4 + i \frac{s^3}{4 \sqrt n} + i \frac{t^3}{4 \sqrt n}-\frac{s^2}{4n}-\frac{t^2}{4n}\\
&\quad+ \frac {7s^4}{32n} +\frac {7t^4}{32n}+\frac{s^2t^2}{48n}+
O\left( \frac{|s|^5+|t|^5}{n^{3/2}} \right).
\end{aligned}
\end{align}
As a next step we want to use the expansion $e^x=1+x+x^2/2+\LandauO(x^3)$ for $x=o(1)$ on the part involving exponents in $s$ and $t$ of order $3$ and higher.  
Therefore we need to split the contour $\gamma_1$ into $3$ parts.
(Remark. At this point the argument in~\cite{DKS2016} is incomplete, but can be repaired in the same way.)

For their definition we need to choose a sequence $\cosc$ such that $\cosc = o(n^{-1/3})$ and $\cosc = \omega(n^{-1/2})$. Thus, we choose $\cosc = n^{-1/2+\nu}$ for $0 < \nu < 1/6$.
Then we define a part $\gamma_{1,1}$ which connects the points $1 \pm i\delta \cosc$ inside the disc of radius $\delta\cosc$ around $1$ and passes $1$ on the left hand side, a part $\gamma_{1,2}$ which connects $1 +i\delta \cosc$ and $1 + i\delta$ by a straight line, and a symmetric part $\gamma_{1,3}$ that connects $1-i\delta \cosc$ and $1-i \delta$ by a straight line, see Figure~\ref{fig:intpath}.

\begin{figure}
	\centering
	\label{fig:intpath}
	\includegraphics[width=4cm]{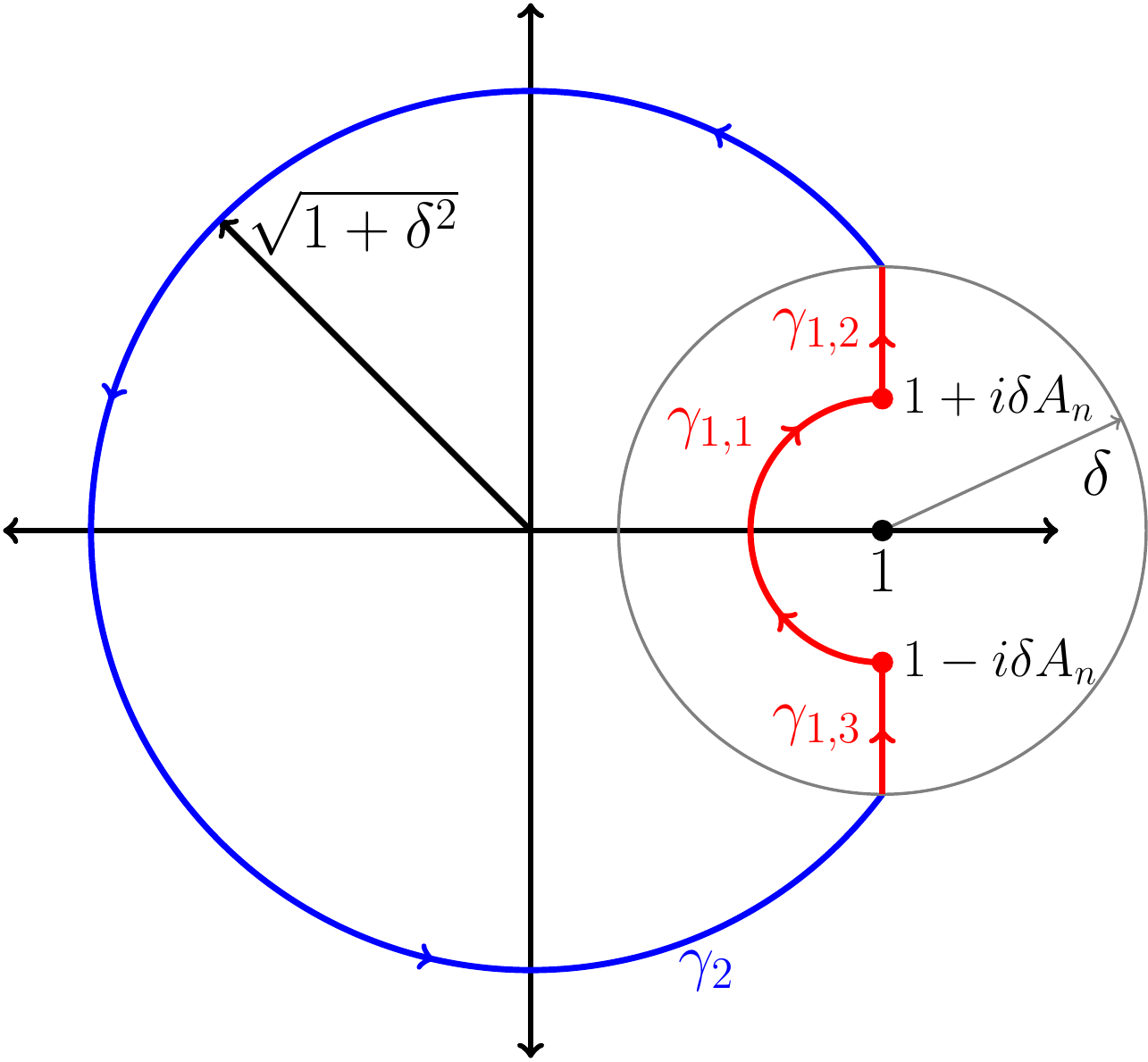}
	\caption{The path of integration used in Lemma~\ref{lem:integral}. The contour consists of two main parts $\gamma_1$ and $\gamma_2$, where $\gamma_1$ is split into $3$ smaller parts $\gamma_{1,1}, \gamma_{1,2}$, and $\gamma_{1,3}$. The asymptotic main contribution arises at $\gamma_{1,1}$. The constant $\delta$ is defined in Lemma~\ref{lem:xexp} and $A_n = n^{-1/2 + \nu}$ with $0<\nu<1/6$.}
\end{figure}

Due to~\eqref{eq:ng} we get the bound
\begin{align*}
	\Re \left(-(n+1)g_n(s,t)\right) &\leq \log(8^{n+1}) - \frac{s^2}{3} - \frac{t^2}{3},
\end{align*}
for large enough $n$. Hence, the integral along $\gamma_{1,2}$ (and also $\gamma_{1,3}$) is negligible as
\begin{align*}
	\int\limits_{\delta n^{\nu} \leq |s|,|t| \leq \delta \sqrt{n}} e^{-(n+1)g_n(s,t)} \, ds \, dt = o\left(8^n e^{-\frac{n^{2\nu}}{3}} \right).
\end{align*}
The lower bound is computed as $A_n\sqrt{n} = n^\nu$, where the choice of $A_n$ is crucial. 

What remains is to treat the integral along $\gamma_{1,1}$. On this part we may use the expansion of $e^x$ to obtain
\begin{align*}
e^{-(n+1)\, g_n(s,t) } &= 8^{n+1} e^{- \frac{s^2}4 - \frac{t^2}4 } 
\left( 1 -\frac{s^2+t^2}{4n}
+i\frac{s^3+t^3}{4\sqrt{n}}
+\frac{7(s^4+t^4)}{32n}
+\frac{s^2t^2}{48n}
\right.
\\&
\left.
-\frac{s^6+t^6}{32n}
-\frac{s^3t^3}{16n}
+O\left( \frac{|s|^5+|s|^7+|t|^5+|t|^7}{n^{3/2}}\right)
\right)
\end{align*}
for $|s|\leq n^{\nu}$ and $|t|\leq \delta n^{\nu}$.
This leads to
\begin{align*}
\hspace{5em}&\hspace{-5em} \frac 1{(2\pi i)^2} \iint\limits_{|s|,|t|\le \delta n^{\nu}, \Im(s),\Im(t)> 0} 
P_n(s,t) e^{-(n+1)\, g_n(s,t)} \frac{\mathrm ds\,\mathrm dt}{st}\\
&= \frac {8^n}{(2\pi i)^2} \iint\limits_{|s|,|t|\le \delta n^{\nu}, \Im(s),\Im(t)> 0}
e^{- \frac{s^2}4 - \frac{t^2}4 }
\left(1+\frac{is^3+it^3}{4\sqrt{n}}
+\frac{7s^4+7t^4}{32n}
+\frac{s^2t^2}{48n}
\right.
\\&
\left.
-\frac{s^6+t^6}{32n}
-\frac{s^3t^3}{16n}
\right)
\frac{\mathrm ds\,\mathrm dt}{st} + O\left( \frac{8^n}{n^{3/2}} \right) \\
&= \frac {8^n}{(2\pi i)^2} \iint\limits_{-\infty <s,t< \infty, \Im(s),\Im(t)> 0}
e^{-\frac{s^2}4 - \frac{t^2}4 }
\left(1+\frac{is^3+it^3}{4\sqrt{n}}
+\frac{7s^4+7t^4}{32n}
+\frac{s^2t^2}{48n}
\right.
\\&
\left.
-\frac{s^6+t^6}{32n}
-\frac{s^3t^3}{16n}
\right)
\frac{\mathrm ds\,\mathrm dt}{st} + O\left( \frac{8^n}{n^{3/2}} \right).
\end{align*}
Finally by writing this as a sum of products of integrals and applying Lemma~\ref{lem:integral} term by term this expression equals
\begin{align*}
&= 8^{n}
\left(\frac 14-\frac 1{2\sqrt{\pi n}}+\frac 1{4\pi n}+\LandauO(n^{-3/2})\right).
\end{align*}
Summing up we arrive at the asymptotics
\[
\frac 1{8^n}[x^n y^{n-1} z^{n-1}]\, F(x,y,z)
=\frac 14 - \frac 1{2\sqrt{\pi n}} + \frac 1{4\pi n}+\LandauO(n^{-3/2}).
\]
By extending the above argument, which is only a computational issue, we obtain more terms in the asymptotic expansion, which yields the statement of Proposition~\ref{prp_diagonal}. For details see the accompanying Maple worksheet~\cite{web}.

\section{Conclusion}
It is an elementary problem to study the behaviour of the digital expansion of an integer under addition of a constant.
More specifically, we wish to understand the \emph{sum of digits} in base $2$ of $n$ and $n+t$, which amounts to study the number of \emph{carries} occurring in the addition of the binary expansions of $n$ and $t$.
The question arises how often a certain number of carries is attained when adding $n$ to a given integer $t$.
At first, this has the appearance of an easy task.
However, we soon meet the difficulty that carries may propagate through several blocks of $1$s; it is not clear how to capture all of the appearing patterns simultaneously.
Both Conjecture~\ref{conj:TD} and Conjecture~\ref{conj:C} concern this question,
and neither of them could be solved for the past seven years since their introduction.
Only partial results have been obtained so far, including an almost-all result for Cusick's conjecture proved by Drmota, Kauers, and the first author.
The current paper adds to our knowledge on the Tu--Deng conjecture by proving an analogous result: Conjecture~\ref{conj:TD} holds almost surely in a precise sense.

Our method certainly can be applied to related questions.
While analoga of~\eqref{eqn_TD_equiv} and~\eqref{eqn_C_equiv} fail for the sum-of-digits function in base $3$, they seem to hold for the Hamming weight of the ternary expansion of $n$ (the number of nonzero digits of $n$ in base $3$).
We are confident that our method yields almost-all results for these questions.

A different kind of extension of the considered problems concerns the sum of digits of $n$, $n+t$ and $n+2t$:
do we have $\lvert \{n\in\{0,\ldots,2^k-1\}: s(n)\leq s(n+t), s(n)\leq s(n+2t)\}\rvert>2^{k-2}$?
Is the same true for $\oplus_k$ instead of $+$?
Again, we expect that nontrivial results can be obtained using our method.

Meanwhile, the full statement of Conjecture~\ref{conj:TD} remains an open problem.
One possible approach to proving it is to assume a hypothetical counterexample to the conjecture, and from it construct a large set of counterexamples, which would contradict the asymptotical statement of our main theorem.
However, it is a nontrivial task to compare the values $P_{t,k}$ for different $t$, in particular to construct (many) integers $t'$ and $k'$ satisfying $P_{t',k'}\geq P_{t,k}$. It follows that this approach cannot yet be used to prove the conjecture.

In a similar vein, we may consider the following approach to proving Conjecture~\ref{conj:C}:
we have numerically $c_{t'}\leq c_t$, where $t'$ is obtained by appending $01\cdots 1$ to the binary expansion and the number of $1$s is large enough.
If this can be proved, we may iterate the procedure of appending $01\cdots 1$, obtaining $t^{(k)}$; moreover, by asymptotic considerations one can certainly prove that $c_{t^{(k)}}>1/2$ for $k$ large enough. By monotonicity, we obtain $c_t>1/2$.
Again, the problem to overcome is the comparison of values of $c_t$ for different $t$, which seems to be difficult.

\medskip

\textbf{Acknowledgments:}
\label{sec:ack}
The first author acknowledges support by the project MuDeRa (Multiplicativity, Determinism, and Randomness),
which is a joint project between the
ANR (Agence Nationale de la Recherche, ANR-14-CE34-0009) and the FWF (Austrian Science Fund, I-1751-N26),
and also by project F5502-N26 (FWF),
which is a part of the Special Research Program
``Quasi Monte Carlo Methods: Theory and Applications''.
The second was supported by the SFB project F50-03 ``Algorithmic and Enumerative Combinatorics''.
Last but not least, we thank Michael Drmota for many insightful discussions, 
and the referees for the detailed feedback and enthusiasm!

\bibliographystyle{siam}
\bibliography{uhu}
\end{document}